\documentclass[11pt,fleqn]{article}
\usepackage[a4paper,margin=2.5cm]{geometry}
\usepackage[numbers,sort&compress]{natbib}
\usepackage{amsmath,amssymb,amsthm,mathtools}
\usepackage{bm}
\usepackage{hyperref}
\newcommand{\R}{\mathbb{R}}
\newcommand{\divg}{\operatorname{div}}
\newcommand{\vol}{\operatorname{vol}}
\newcommand{\dist}{\operatorname{dist}}
\newcommand{\osc}{\operatorname{osc}}
\newcommand{\Lie}{\mathcal{L}}
\newcommand{\A}{\mathcal{A}}
\newcommand{\ThetaSet}{\Theta}

\theoremstyle{plain}
\newtheorem{theorem}{Theorem}
\newtheorem{proposition}{Proposition}
\newtheorem{corollary}{Corollary}

\theoremstyle{definition}

\theoremstyle{remark}
\newtheorem{remark}{Remark}

\title{\textbf{Convex Density Design for Regional Phase-Volume Contraction of Nonlinear Systems}}

\author{Igor B. Furtat\thanks{Institute for Problems in Mechanical Engineering,
Russian Academy of Sciences, Boljshoy prospekt V.O., 61,
Saint Petersburg, 199178, Russia.
Email: \texttt{cainenash@mail.ru}.
ORCID: 0000-0003-4679-5884.}}

\date{}

\begin{document}

\maketitle

\begin{abstract}
Regional weighted phase-volume contraction for nonlinear systems is studied
using positive densities selected from a finite-dimensional log-affine family.
The affine dependence of the density divergence on the design parameters turns
density selection into a semi-infinite convex optimization problem that
maximizes a certified uniform contraction margin on a compact region. This
provides a constructive alternative to testing a preselected density and
ensures a well-posed design problem under compact parameter constraints.
Finite active-point and minimax characterizations identify the worst-case
states governing the optimum, while rigorous sampling bounds allow continuum
contraction guarantees to be recovered from finitely many inequalities. For
scalar power densities, feasibility is characterized by an exact interval
condition. The same structure is preserved in discrete time and admits
distance-weighted and conformal-geometric interpretations. Overall, the
framework provides computable certificates for detecting regional contraction
that may remain hidden from ordinary divergence. A two-parameter limit-cycle
example demonstrates this effect when both ordinary phase volume and a natural
single-basis density fail.
\end{abstract}


\maketitle

\section{Introduction}\label{sec:intro}

The divergence of a vector field determines the infinitesimal change of ambient phase volume under its flow. This classical Liouville viewpoint is complementary to trajectory-wise Lyapunov analysis \cite{Lyapunov,Liouville,Arnold}. A major development of the volume viewpoint is Rantzer's density-function approach \cite{Rantzer01}, closely related in the planar case to the classical Bendixson--Dulac idea \cite{Dulac23}, in which the sign of $\divg(\rho f)$ for a positive density $\rho$ yields almost-global convergence information under suitable hypotheses. Converse and optimization aspects were subsequently studied in, e.g., \cite{Rantzer02,RantzerHedlund03}. The convex structure of density inequalities was recognized early in nonlinear stabilization \cite{RantzerParrilo00} and developed in controller and density synthesis, including sum-of-squares formulations \cite{PrajnaParriloRantzer04,Masubuchi10} and data-driven synthesis \cite{ChoiVaidyaChen21}. Density criteria have also been extended to time-varying and periodically time-varying systems \cite{MasubuchiKikuchi21}. Density and Lyapunov-like criteria for discrete-time almost-global convergence have been proposed in \cite{Kamaldar24}, while a Perron--Frobenius semigroup characterization of almost-global attractivity has been established in \cite{Lorenzetti25}.

A distinct but related line of work studies contraction of geometric objects transported by nonlinear dynamics. Classical contraction theory concerns one-dimensional infinitesimal displacements, whereas $k$-contraction studies decay of $k$-dimensional volumes and restrictions on asymptotic behavior \cite{Wu22}. Recent design-oriented conditions for $k$-contraction further emphasize the current interest in constructive geometric contraction certificates \cite{Cecilia26}. The present study complements density-function stability theory and $k$-contraction by focusing on a finite-dimensional density-design problem. Specifically, the parameters of a prescribed low-dimensional family of positive reference densities are selected so as to maximize a certified contraction rate of weighted $n$-dimensional phase volume over a given region.

Several recent papers use density-like objects for different purposes. Density-dependent dynamical systems and adaptive control are studied in \cite{FurtatTAC26}, while weighted phase volumes combined with time rescaling and reachable-set estimates are considered in \cite{Furtat26Arc}. In contrast, the vector field is not modified or time-rescaled here. The density only changes the reference volume form, and the contribution is the constrained selection of that density inside the fixed log-affine family
\begin{equation}
\label{eq:rho-family-intro}
\rho_\theta(x)=\exp\!\big(\theta^\top\psi(x)\big),
\qquad \theta\in \R^p.
\end{equation}
The basis functions $\psi_1,\ldots,\psi_p$ may encode geometric or physical information, while only the finite-dimensional coefficient vector $\theta$ is optimized. The scalar power family $\rho_\alpha=w^\alpha$ corresponds to $p=1$ and $\psi=\ln w$. Allowing an arbitrary positive density would recover the general density-function search. Fixing the basis converts density selection into a transparent low-dimensional regional design problem.

The affine identity provides the basis for the proposed framework. Its simplicity is precisely what makes the density-design problem tractable: it converts the selection of a positive reference density into a finite-dimensional constrained max-min problem and delivers a complete set of density-specific certificates--convexity of the rate-feasible parameter set, concavity and Lipschitz continuity of the contraction margin, finite active-point optimality conditions, an exact finite-dimensional minimax dual, and an explicit $\varepsilon$-net bound for rigorous sampling. These finite-dimensional consequences constitute the main contribution of the paper. For the considered flow, the local logarithmic rate of change of the density-weighted volume element depends affinely on the density-design parameters. This affine dependence provides the structural basis for the convex formulation developed below. Maximizing a uniform contraction rate on a compact region therefore becomes a semi-infinite convex optimization problem. Semi-infinite programming is a classical optimization framework and has recently appeared in other nonlinear-control problems, including minimal-time control \cite{OustryTacchi26}. Here it arises from restricted density selection and leads to explicit density-design certificates. We establish five consequences useful for certification and computation: (i) rate-feasible parameters form a convex set, with an exact scalar interval formula; (ii) the constrained optimum exists and the margin is concave and Lipschitz; (iii) optimality is equivalent to a finite active-point certificate; (iv) the optimal value has an exact finite-dimensional minimax dual and finite sampling becomes rigorous through an explicit $\varepsilon$-net error bound; (v) the same affine design structure holds for discrete-time maps through the logarithm of the density-weighted Jacobian.

The paper is deliberately formulated at the level of weighted-volume certificates. Decay of weighted volume does not by itself imply Lyapunov stability or attraction of every trajectory. Additional hypotheses are required for pointwise conclusions. This distinction is essential in density-based methods \cite{Rantzer01,Lorenzetti25}. We therefore separate set-wise volume contraction from a squared-distance moment estimate near invariant submanifolds and state a distinct condition when pointwise distance contraction is desired.

The rest of the paper is organized as follows. Section~\ref{sec:setup} gives the transport identity. Section~\ref{sec:design} develops the convex density-design problem, active-point and minimax certificates, rigorous finite sampling, and the scalar exponent characterization. Section~\ref{sec:volume} gives continuous-time volume consequences. Section~\ref{sec:discrete} treats maps and Poincar\'e sections. Section~\ref{sec:submanifold} studies squared-distance moments near invariant submanifolds. Section~\ref{sec:geometry} gives the conformal interpretation. Section~\ref{sec:example} presents an analytically solvable two-parameter example in which independent density shaping succeeds whereas a natural one-basis restriction does not.

\section{Parameterized densities and transport identity}\label{sec:setup}

Let $(M,g)$ be an oriented $n$-dimensional Riemannian manifold with volume form $dV_g$. Consider
\begin{equation*}
\label{eq}
\dot x=f(x),
\end{equation*}
where $f\in C^1(M,TM)$, $TM$ denotes the tangent bundle of $M$, and $f(x)\in T_xM$ for each $x\in M$. Assume that $f$ generates a forward-complete $C^1$ semiflow $\Phi_t$, $t\ge0$. Let $\psi=(\psi_1,\ldots,\psi_p)^\top\in C^1(M,\R^p)$ and define the positive density \eqref{eq:rho-family-intro}.
For a measurable set $D$ with compact closure, define
\begin{equation*}
\label{eq:mu}
\mu_\theta(D):=\int_D\rho_\theta\,dV_g.
\end{equation*}
For a scalar function $h$, write $\Lie_fh=dh(f)=\langle\nabla_gh,f\rangle_g$. Set
\begin{equation*}
\label{eq:bdef}
a(x):=\divg_g f(x),\qquad
b(x):=\Lie_f\psi(x)\in\R^p.
\end{equation*}

The evolution of the density-weighted measure under the flow is characterized by the following transport identity.
\begin{proposition}
\label{prop:transport}
For every measurable set $D\subset M$ with compact closure,
\begin{equation}\label{eq:transport}
\frac{d}{dt}\mu_\theta(\Phi_t(D))
=\int_{\Phi_t(D)}D_\theta f\,\rho_\theta\,dV_g,
\end{equation}
where
\begin{equation}
\label{eq:Dtheta}
D_\theta f
:=\divg_g f+\theta^\top\Lie_f\psi
=\rho_\theta^{-1}\divg_g(\rho_\theta f).
\end{equation}
\end{proposition}

\begin{proof}
Let $J_t(x)$ be the positive Riemannian Jacobian of $\Phi_t$. Change of variables gives
\[
\mu_\theta(\Phi_t(D))
=\int_D \rho_\theta(\Phi_t(x))J_t(x)\,dV_g(x).
\]
Along the flow, $\dot J_t=(\divg_g f)\circ\Phi_t\,J_t$. Since $D$ has compact closure, differentiation under the integral sign is justified on every bounded time interval on which the image remains in a compact flow tube. Hence
\[
\frac{d}{dt}\mu_\theta(\Phi_t(D))
=\int_D\big(\Lie_f\rho_\theta+\rho_\theta\divg_gf\big)(\Phi_t(x))J_t(x)\,dV_g(x),
\]
and changing variables back yields \eqref{eq:transport}. Finally,
$\Lie_f\rho_\theta=\rho_\theta\theta^\top\Lie_f\psi$, and the last equality in \eqref{eq:Dtheta} is the product rule for divergence.
\end{proof}

\begin{remark}
\label{rem:power}
For $p=1$, let $w>0$ be $C^1$ and choose $\psi=\ln w$, $\theta=\alpha\in\R$. Then
\[
\rho_\alpha=w^\alpha,\qquad
D_\alpha f=\divg_gf+\alpha\Lie_f\ln w
=w^{-\alpha}\divg_g(w^\alpha f).
\]
Thus the power-weighted phase volume is a one-parameter subfamily of the density-divergence construction.
\end{remark}

\section{Density shaping as a convex design problem}\label{sec:design}

Let $K\subset M$ be nonempty and compact. The design objective is to make $D_\theta f$ uniformly negative on $K$. Define the certified contraction margin
\begin{equation}\label{eq:q}
q_K(\theta):=\min_{x\in K}\big[-D_\theta f(x)\big]
=\min_{x\in K}\big[-a(x)-\theta^\top b(x)\big].
\end{equation}
The minimum exists because $a$ and $b$ are continuous on $K$.
A parameter $\theta$ certifies uniform weighted-volume contraction on $K$ precisely when $q_K(\theta)>0$. 
In that case any $0<\lambda\le q_K(\theta)$ is a certified contraction rate. Values $q_K(\theta)\le0$ are retained as signed margins but do not certify contraction.

The following theorem establishes the geometric structure of the density-design problem and underpins all subsequent optimality and computational certificates: concavity and Lipschitz continuity of the contraction margin, convexity and compactness of the rate-feasible parameter sets, attainment of the constrained optimum, and preservation of convexity under density-conditioning constraints.

\begin{theorem}
\label{th:design}
Let $K\subset M$ be nonempty and compact and let $\ThetaSet\subset\R^p$ be a nonempty compact convex set. Then:
\begin{enumerate}
\item $q_K$ is concave on $\R^p$. Moreover, with $B_K:=\max_{x\in K}\|b(x)\|_2$,
\begin{equation}
\label{eq:lipschitz}
|q_K(\theta_1)-q_K(\theta_2)|\le B_K\|\theta_1-\theta_2\|_2.
\end{equation}
\item For every prescribed rate $\lambda\in\R$, the admissible parameter set
\begin{equation*}
\label{eq:A}
\A_K(\lambda):=\{\theta\in\ThetaSet:D_\theta f(x)\le-\lambda\ \forall x\in K\}
\end{equation*}
is compact and convex, and admits the half-space representation
\begin{equation}\label{eq:halfspaces}
\A_K(\lambda)=\ThetaSet\cap
\bigcap_{x\in K}\{\theta:\theta^\top b(x)\le-\lambda-a(x)\}.
\end{equation}
\item The optimal constrained margin
\begin{equation}\label{eq:qstar}
q_K^*:=\max_{\theta\in\ThetaSet}q_K(\theta)
\end{equation}
is attained. Equivalently, $(\theta^*,q_K^*)$ solves the semi-infinite convex program
\begin{equation}\label{eq:sip}
\begin{aligned}
\text{maximize}_{\theta\in\ThetaSet,\,\lambda\in\R}\quad &\lambda\\
\text{subject to}\quad &a(x)+\theta^\top b(x)+\lambda\le0,
\qquad x\in K.
\end{aligned}
\end{equation}
\item For every $\kappa_\rho\ge1$, the conditioning-restricted set
\begin{equation}\label{eq:Theta-kappa}
\ThetaSet_{\kappa_\rho}
:=\{\theta\in\ThetaSet:\osc_K(\theta^\top\psi)\le\ln\kappa_\rho\}
\end{equation}
is compact and convex. If $\ThetaSet_{\kappa_\rho}$ is nonempty, the max-margin problem remains a convex program and attains its optimum after imposing this distortion bound.
\end{enumerate}
\end{theorem}

\begin{proof}
For each fixed $x\in K$, the function $-a(x)-\theta^\top b(x)$ is affine in $\theta$. The pointwise minimum of affine functions is concave, proving the first statement in item 1. For any $\theta_1,\theta_2$,
\[
-a(x)-\theta_1^\top b(x)
\ge -a(x)-\theta_2^\top b(x)-B_K\|\theta_1-\theta_2\|_2.
\]
Taking the minimum over $x$ and then interchanging $\theta_1,\theta_2$ gives \eqref{eq:lipschitz}. Hence $q_K$ is continuous.

Equation \eqref{eq:halfspaces} follows directly from the pointwise inequalities. It is an intersection of the compact convex set $\ThetaSet$ with closed half-spaces, so it is compact and convex. Finally, continuity of $q_K$ and compactness of $\ThetaSet$ imply attainment of \eqref{eq:qstar}. The hypograph formulation of the max-min problem is exactly \eqref{eq:sip}. All constraints are affine in $(\theta,\lambda)$. Finally, $\theta\mapsto\osc_K(\theta^\top\psi)$ is convex because it is the sum of the convex support function $\max_K\theta^\top\psi$ and the convex function $\max_K(-\theta^\top\psi)$. Therefore \eqref{eq:Theta-kappa} is a closed convex subset of the compact set $\ThetaSet$.
\end{proof}

The semi-infinite formulation also admits a finite first-order certificate. Define the active set
\begin{equation*}
\label{eq:active}
X(\theta):=\{x\in K:-a(x)-\theta^\top b(x)=q_K(\theta)\}.
\end{equation*}

A distinctive feature of the semi-infinite convex program is that its global optimality admits an exact finite-dimensional characterization, i.e. it is necessary and sufficient to verify a balance condition on finitely many worst-case states. This yields the following statement.

\begin{proposition}
\label{prop:active}
Let the assumptions of Theorem~\ref{th:design} hold. A point $\bar\theta\in\ThetaSet$ maximizes $q_K$ over $\ThetaSet$ if and only if there exist active points $x_1,\ldots,x_m\in X(\bar\theta)$, with $m\le p+1$, and coefficients $\nu_i\ge0$, $\sum_{i=1}^m\nu_i=1$, such that
\begin{equation}\label{eq:normal-cert}
-\sum_{i=1}^m\nu_i b(x_i)\in N_{\ThetaSet}(\bar\theta),
\end{equation}
where
$N_{\ThetaSet}(\bar\theta)
:=\{v:v^\top(\theta-\bar\theta)\le0\ \text{for all }\theta\in\ThetaSet\}$
is the normal cone. In particular, if $\bar\theta$ lies in the interior of $\ThetaSet$, optimality is equivalent to the existence of active points and weights satisfying
\begin{equation}\label{eq:balance}
\sum_{i=1}^m\nu_i b(x_i)=0.
\end{equation}
\end{proposition}

\begin{proof}
Suppose first that $\bar\theta$ maximizes $q_K$. Because $K$ is compact and the affine integrand in \eqref{eq:q} is jointly continuous in $(x,\theta)$, Danskin's theorem for a pointwise minimum gives
\[
\partial^+q_K(\bar\theta)
=\operatorname{co}\{-b(x):x\in X(\bar\theta)\},
\]
see the original max--min result in \cite{Danskin66}. The first-order condition for maximizing a finite concave function over a closed convex set is
$\partial^+q_K(\bar\theta)\cap N_{\ThetaSet}(\bar\theta)\ne\varnothing$, see, e.g., \cite{BoydVandenberghe04}. Thus some convex combination of active supergradients belongs to the normal cone. By Carath\'eodory's theorem, at most $p+1$ active points are required, which proves \eqref{eq:normal-cert}.

Conversely, suppose active points and weights satisfy \eqref{eq:normal-cert}, and set
\[
g:=-\sum_{i=1}^m\nu_i b(x_i)\in N_{\ThetaSet}(\bar\theta).
\]
For every $\theta\in\ThetaSet$, the definition of $q_K$ and activity of the $x_i$ give
\begin{align*}
q_K(\theta)
&\le \sum_{i=1}^m\nu_i[-a(x_i)-\theta^\top b(x_i)]\\
&=q_K(\bar\theta)+g^\top(\theta-\bar\theta)
\le q_K(\bar\theta),
\end{align*}
where the last inequality is the normal-cone condition. Hence $\bar\theta$ is a global maximizer. If $\bar\theta$ is interior to $\ThetaSet$, then $N_{\ThetaSet}(\bar\theta)=\{0\}$, so \eqref{eq:normal-cert} reduces to \eqref{eq:balance}. 
The converse just proved shows that this balance condition is also sufficient.
\end{proof}

The max--min problem also has an exact finite-dimensional dual representation. Define
\begin{equation*}
\label{eq:CK}
\mathcal C_K:=\operatorname{co}\{(a(x),b(x)):x\in K\}\subset\R^{p+1}
\end{equation*}
and let
$\sigma_{\ThetaSet}(y):=\max_{\theta\in\ThetaSet}\theta^\top y$
be the support function of the compact convex parameter set.

Beyond the primal active-point certificate, the semi-infinite program admits an exact finite-dimensional dual representation. This dual formulation provides a complementary global optimality condition and reveals that the dual optimizer is supported on finitely many active states, yielding the following result.

\begin{proposition}
\label{prop:minimax}
Under the assumptions of Theorem~\ref{th:design},
\begin{equation}\label{eq:minimax-dual}
q_K^*=\min_{(A,B)\in\mathcal C_K}
\big[-A+\sigma_{\ThetaSet}(-B)\big].
\end{equation}
The minimum is attained. Moreover, for any primal optimizer $\theta^*$, an optimal pair $(A^*,B^*)$ admits a representation
\begin{equation}\label{eq:finite-dual-support}
(A^*,B^*)=\sum_{i=1}^m\nu_i\big(a(x_i),b(x_i)\big),
~~~ m\le p+2,
\end{equation}
with $x_i\in X(\theta^*)$, $\nu_i>0$, and $\sum_i\nu_i=1$.
\end{proposition}

\begin{proof}
Because the image of $K$ under $x\mapsto(a(x),b(x))$ is compact, its convex hull $\mathcal C_K$ is compact in the finite-dimensional space $\R^{p+1}$. For
\[
F(\theta,A,B):=-A-\theta^\top B,
\]
the minimum over the generating set equals the minimum over its convex hull, since $F$ is affine in $(A,B)$. Hence
\[
q_K^*=\max_{\theta\in\ThetaSet}\min_{(A,B)\in\mathcal C_K}F(\theta,A,B).
\]
Both feasible sets are compact and convex, and $F$ is continuous and affine in each argument. Sion's minimax theorem \cite{Sion58} therefore yields
\[
q_K^*=\min_{(A,B)\in\mathcal C_K}\max_{\theta\in\ThetaSet}F(\theta,A,B),
\]
which is \eqref{eq:minimax-dual}. Attainment follows from compactness. Carath\'eodory's theorem in $\R^{p+1}$ gives the representation \eqref{eq:finite-dual-support} with at most $p+2$ points.

For the finite-support statement, choose a dual optimizer $(A^*,B^*)$ and the prescribed primal optimizer $\theta^*$. The minimax equality implies
$F(\theta^*,A^*,B^*)=q_K^*$. By Carath\'eodory, represent $(A^*,B^*)$ by at most $p+2$ points and discard zero coefficients. Since
$-a(x)-{\theta^*}^\top b(x)\ge q_K^*$ for every $x\in K$, while their positive weighted average equals $q_K^*$, every supporting point must satisfy equality and hence belongs to $X(\theta^*)$.
\end{proof}

For numerical verification, a finite net can replace the continuum $K$ while retaining an explicit error bound.

\begin{proposition}
\label{prop:grid}
Assume $K$ is a compact metric space and, on $K$,
\[
|a(x)-a(y)|\le L_a d(x,y),~~~
\|b(x)-b(y)\|_2\le L_b d(x,y).
\]
Let $\varepsilon>0$ and let $G_\varepsilon\subset K$ be a finite $\varepsilon$-net. Then, for every $\theta\in\R^p$,
\begin{equation}\label{eq:grid-bound}
\begin{split}
q_K(\theta)\ge{}&
\min_{z\in G_\varepsilon}[-a(z)-\theta^\top b(z)]\\
&-(L_a+\|\theta\|_2L_b)\varepsilon.
\end{split}
\end{equation}
Consequently, if
\begin{equation}\label{eq:grid-cert}
a(z)+\theta^\top b(z)
\le-\lambda-(L_a+\|\theta\|_2L_b)\varepsilon
\end{equation}
for all $z\in G_\varepsilon$, then $D_\theta f\le-\lambda$ on all of $K$.
\end{proposition}

\begin{proof}
For any $x\in K$, choose $z\in G_\varepsilon$ with $d(x,z)\le\varepsilon$. Cauchy--Schwarz and the Lipschitz assumptions give
\[
-a(x)-\theta^\top b(x)
\ge-a(z)-\theta^\top b(z)
-(L_a+\|\theta\|_2L_b)\varepsilon.
\]
Taking the minimum over $x$ proves \eqref{eq:grid-bound}, \eqref{eq:grid-cert} then implies the continuum inequality.
\end{proof}

\begin{remark}
\label{rem:computation}
For design over compact $\ThetaSet$, let $R_\Theta:=\max_{\theta\in\ThetaSet}\|\theta\|_2$. Replacing $\|\theta\|_2$ in \eqref{eq:grid-cert} by $R_\Theta$ gives the uniform sufficient constraints
\[
a(z)+\theta^\top b(z)
\le-\lambda-(L_a+R_\Theta L_b)\varepsilon,
\qquad z\in G_\varepsilon.
\]
These constraints are affine in $(\theta,\lambda)$. Hence, for polyhedral $\ThetaSet$, the sampled certified max-margin problem is a linear program. Retaining the sharper parameter-dependent norm in \eqref{eq:grid-cert} instead gives a convex norm-constrained problem. Proposition~\ref{prop:grid} supplies the continuum guarantee, and the grid may be refined where the sampled residual is smallest.
\end{remark}

\begin{remark}
\label{rem:constraint}
Without a restriction on $\theta$, the supremum of $q_K(\theta)$ need not be finite. For example, if a scalar $b(x)=\Lie_f\psi(x)$ is uniformly positive on $K$, then $q_K(\alpha)\to+\infty$ as $\alpha\to-\infty$. This does not represent arbitrarily fast physical contraction of trajectories. It reflects an increasingly distorted reference measure. A compact design set $\ThetaSet$ therefore has both a mathematical and an interpretive role.

One useful additional restriction is a bound on density conditioning,
\begin{equation}\label{eq:conditioning}
\frac{\max_K\rho_\theta}{\min_K\rho_\theta}
=\exp\!\big(\osc_K(\theta^\top\psi)\big)\le\kappa_\rho.
\end{equation}
As shown in Theorem~\ref{th:design}, \eqref{eq:conditioning} is a convex constraint. It does not, in general, make the parameter set compact: multiplication of a density by a positive constant leaves $D_\theta f$ unchanged, and redundant basis directions may therefore remain unbounded. Thus \eqref{eq:conditioning} should be combined with a compact box or ellipsoid, or with an explicit normalization that removes such gauge directions.
\end{remark}

The scalar power family admits a closed-form characterization.

\begin{corollary}
\label{cor:alpha}
Let $K$ be nonempty and compact and let $w>0$ be $C^1$ on a neighborhood of $K$. Set
\[
a(x)=\divg_gf(x),\qquad b(x)=\Lie_f\ln w(x).
\]
For a prescribed $\lambda\in\R$, define the extended-real bounds
\begin{align*}
L_\lambda&:=\sup_{\{x\in K:b(x)<0\}}
\frac{-\lambda-a(x)}{b(x)}\in[-\infty,+\infty],\notag\\
U_\lambda&:=\inf_{\{x\in K:b(x)>0\}}
\frac{-\lambda-a(x)}{b(x)}\in[-\infty,+\infty],\label{eq:LU}
\end{align*}
with $L_\lambda=-\infty$ when $\{b<0\}$ is empty and $U_\lambda=+\infty$ when $\{b>0\}$ is empty. If
\begin{equation}\label{eq:zero-compat}
a(x)\le-\lambda\qquad\text{whenever }b(x)=0,
\end{equation}
then the set of \emph{finite} admissible exponents is exactly
\begin{equation}\label{eq:alpha-interval}
\begin{split}
\mathcal I_\lambda
&:=\{\alpha\in\R:D_\alpha f\le-\lambda\text{ on }K\}\\
&=\{\alpha\in\R:L_\lambda\le\alpha\le U_\lambda\}.
\end{split}
\end{equation}
If \eqref{eq:zero-compat} fails, then $\mathcal I_\lambda=\varnothing$. Under \eqref{eq:zero-compat}, the set in \eqref{eq:alpha-interval} is nonempty if and only if
\begin{equation}\label{eq:interval-nonempty}
L_\lambda\le U_\lambda,\qquad
L_\lambda<+\infty,\qquad U_\lambda>-\infty.
\end{equation}
For a constrained interval $\ThetaSet=[\alpha_{\min},\alpha_{\max}]$, the feasible exponents are $\mathcal I_\lambda\cap\ThetaSet$.
\end{corollary}

\begin{proof}
At a point where $b(x)>0$, the inequality $a(x)+\alpha b(x)\le-\lambda$ is equivalent to
$\alpha\le(-\lambda-a(x))/b(x)$; taking all such points gives $\alpha\le U_\lambda$. At a point where $b(x)<0$, it is equivalent to the lower bound $\alpha\ge(-\lambda-a(x))/b(x)$, hence $\alpha\ge L_\lambda$. Points with $b(x)=0$ impose exactly \eqref{eq:zero-compat}. Therefore the finite feasible exponents are precisely \eqref{eq:alpha-interval}. Such a real number exists exactly under \eqref{eq:interval-nonempty}.
\end{proof}

\section{Weighted-volume contraction and invariant sets}\label{sec:volume}

We now connect the design margin to transported sets. Let $K$ be a compact positively invariant region, i.e. $\Phi_t(K)\subseteq K$ for $t\ge0$.

\begin{theorem}
\label{th:contraction}
If, for some $\theta$ and $\lambda>0$,
\begin{equation}\label{eq:Dnegative}
D_\theta f(x)\le-\lambda,\qquad x\in K,
\end{equation}
then every measurable set $D\subset K$ satisfies
\begin{equation}\label{eq:exp-contract}
\mu_\theta(\Phi_t(D))\le e^{-\lambda t}\mu_\theta(D),
\qquad t\ge0.
\end{equation}
In particular, any $\theta\in\A_K(\lambda)$ from Theorem~\ref{th:design} certifies the rate $\lambda$.
\end{theorem}

\begin{proof}
Positive invariance gives $\Phi_t(D)\subseteq K$. Proposition~\ref{prop:transport} and \eqref{eq:Dnegative} imply
\[
\frac{d}{dt}\mu_\theta(\Phi_t(D))
\le-\lambda\mu_\theta(\Phi_t(D)).
\]
Gronwall's inequality yields \eqref{eq:exp-contract}.
\end{proof}

\begin{remark}
\label{rem:nullsets}
On compact $K$, the positive continuous density satisfies
$0<m_\theta\le\rho_\theta\le M_\theta<\infty$. Hence, for measurable $E\subset K$,
\begin{equation}\label{eq:equivmeas}
m_\theta\vol_g(E)\le\mu_\theta(E)\le M_\theta\vol_g(E).
\end{equation}
Therefore weighted and Riemannian measures have exactly the same null sets on $K$. The value of density shaping is not a new notion of ``zero volume'' on compact regions. It is the possibility of obtaining a useful sign and rate in the transport equation when the ordinary divergence does not provide one.
\end{remark}

The following consequence corrects a common overstatement: a nonpositive density divergence does not exclude all invariant sets, but under a dense strictness condition it excludes invariant Jordan sets of positive ambient volume.

Here a compact set is called Jordan measurable if its boundary has zero Riemannian volume.

\begin{proposition}
\label{prop:invariant}
Let $O\subset M$ be open. Suppose $D_\theta f\le0$ on $O$ and the set
$\{x\in O:D_\theta f(x)<0\}$ is dense in $O$. Then every compact Jordan-measurable set $A\subset O$ satisfying $\Phi_t(A)=A$ for all $t\ge0$ satisfies
\begin{equation*}
\label{eq:zero-measure}
\mu_\theta(A)=\vol_g(A)=0.
\end{equation*}
\end{proposition}

\begin{proof}
Assume, to the contrary, that $\mu_\theta(A)>0$. By \eqref{eq:equivmeas} on the compact set $A$, $\vol_g(A)>0$. Since $A$ is Jordan measurable, its boundary has zero volume. If $\operatorname{int}A$ were empty, compactness would imply $A=\partial A$ and hence $\vol_g(A)=0$, a contradiction. Thus $\operatorname{int}A\ne\varnothing$.

Invariance and Proposition~\ref{prop:transport}, applied to the compact measurable set $A$, give at $t=0$
\[
0=\int_A D_\theta f\,\rho_\theta\,dV_g.
\]
The integrand is nonpositive, so $D_\theta f=0$ almost everywhere on $A$. By continuity, it must vanish everywhere on the open set $\operatorname{int}A$: otherwise a point with a negative value would have a neighborhood of positive volume on which the integrand is strictly negative. But density of $\{D_\theta f<0\}$ forces this set to intersect $\operatorname{int}A$, a contradiction.
\end{proof}

\section{Discrete-time systems and Poincar\'e maps}\label{sec:discrete}

The same design structure survives discretization. Let $P:O\to P(O)\subset M$ be a $C^1$ diffeomorphism onto its image, where $O\subset M$ is open. Let $J_P(x)>0$ denote its positive Riemannian Jacobian (the absolute Jacobian with respect to $dV_g$), and let $K\subset O$ be compact and positively invariant, i.e. $P(K)\subseteq K$. Define
\begin{equation*}
\label{eq:weighted-jac}
J_\theta P(x):=
J_P(x)\frac{\rho_\theta(P(x))}{\rho_\theta(x)}
\end{equation*}
and its logarithm
\begin{equation}\label{eq:logjac}
\ell_\theta(x):=\ln J_\theta P(x)
=\ln J_P(x)+\theta^\top\big(\psi(P(x))-\psi(x)\big).
\end{equation}
Thus $\ell_\theta$ is affine in the same design parameter.

\begin{theorem}
\label{th:discrete}
For every measurable $D\subset K$,
\begin{equation}\label{eq:change-map}
\mu_\theta(P(D))=\int_D J_\theta P(x)\rho_\theta(x)\,dV_g(x).
\end{equation}
If, for some $\gamma>0$,
\begin{equation}\label{eq:disc-cond}
\ell_\theta(x)\le-\gamma,\qquad x\in K,
\end{equation}
then
\begin{equation}\label{eq:disc-contract}
\mu_\theta(P^k(D))\le e^{-k\gamma}\mu_\theta(D),
\qquad k\in\mathbb N.
\end{equation}
Moreover, the discrete design margin
\begin{equation*}
\label{eq:qP}
q_K^P(\theta):=\min_{x\in K}[-\ell_\theta(x)]
\end{equation*}
is concave and Lipschitz, and maximizing it over a compact convex $\ThetaSet$ is again a semi-infinite convex program.
\end{theorem}

\begin{proof}
Equation \eqref{eq:change-map} is the change-of-variables formula. Under \eqref{eq:disc-cond}, $J_\theta P\le e^{-\gamma}$, hence
$\mu_\theta(P(D))\le e^{-\gamma}\mu_\theta(D)$. Positive invariance permits iteration and gives \eqref{eq:disc-contract}. Finally, \eqref{eq:logjac} is affine in $\theta$, so the proof of Theorem~\ref{th:design} applies verbatim with $a(x)=\ln J_P(x)$ and $b(x)=\psi(P(x))-\psi(x)$.
\end{proof}

\begin{corollary}
\label{cor:discrete-invariant}
Under the hypotheses of Theorem~\ref{th:discrete}, suppose a compact measurable set $B\subset K$ satisfies $P(B)\subseteq B$ and \eqref{eq:disc-cond}. Then
\begin{equation*}\label{eq:Ainf}
A_\infty:=\bigcap_{k=0}^{\infty}P^k(B)
\end{equation*}
has $\mu_\theta(A_\infty)=0$ and, consequently, $\vol_g(A_\infty)=0$.
\end{corollary}

\begin{proof}
The sets $P^k(B)$ are nested and have finite $\mu_\theta$-measure. By \eqref{eq:disc-contract},
$\mu_\theta(P^k(B))\le e^{-k\gamma}\mu_\theta(B)\to0$. Continuity of finite measures from above gives $\mu_\theta(A_\infty)=0$, and \eqref{eq:equivmeas} gives the Riemannian null-measure statement. Moreover, $P(A_\infty)=A_\infty$: the inclusion $P(A_\infty)\subseteq A_\infty$ is immediate, while for $y\in A_\infty$ the unique $x\in B$ with $P(x)=y$ belongs to every $P^k(B)$ by injectivity of $P$, hence $x\in A_\infty$.
\end{proof}

\begin{remark}
Theorem~\ref{th:discrete} applies directly to a Poincar\'e return map on a transverse section whenever the return map is well defined and injective on the region of interest. It is intentionally stated as a discrete-time theorem rather than obtained by formally applying a continuous-time flow theorem to a map.
\end{remark}

\section{Distance-weighted moments near invariant submanifolds}\label{sec:submanifold}

Let $N\subset M$ be a compact embedded invariant submanifold without boundary. On a sufficiently small tubular neighborhood $U$ of $N$, the squared Riemannian distance
\begin{equation*}\label{eq:delta}
\delta_N(x):=\dist_g(x,N)^2
\end{equation*}
is smooth. This avoids the nonsmoothness of the unsquared distance on $N$. Define the distance-weighted moment of a measurable set $D\subset U$ by
\begin{equation*}\label{eq:moment}
\mathcal M_{N,\theta}(D):=
\int_D\delta_N(x)\rho_\theta(x)\,dV_g(x).
\end{equation*}
This is an ambient weighted moment measuring how much transported volume lies away from $N$. It is not identified with a geometric transverse volume.

\begin{theorem}
\label{th:distance}
Assume $U$ is positively invariant and, for some $\lambda>0$,
\begin{equation}\label{eq:moment-cond}
\Lie_f\delta_N(x)+\delta_N(x)D_\theta f(x)
\le-\lambda\delta_N(x),
\qquad x\in U.
\end{equation}
Then every measurable $D\Subset U$ satisfies
\begin{equation}\label{eq:moment-contract}
\mathcal M_{N,\theta}(\Phi_t(D))
\le e^{-\lambda t}\mathcal M_{N,\theta}(D).
\end{equation}
If, independently,
\begin{equation}\label{eq:pointwise-distance}
\Lie_f\delta_N\le-2\eta\delta_N
\quad\text{on }U
\end{equation}
for some $\eta>0$, then
\begin{equation}\label{eq:distance-exp}
\dist_g(\Phi_t(x),N)
\le e^{-\eta t}\dist_g(x,N),\qquad x\in U.
\end{equation}
\end{theorem}

\begin{proof}
Applying the Proposition \ref{prop:transport} to the smooth nonnegative weight $\delta_N\rho_\theta$ gives
\[
\frac{d}{dt}\mathcal M_{N,\theta}(\Phi_t(D))
=\int_{\Phi_t(D)}\rho_\theta
\big(\Lie_f\delta_N+\delta_ND_\theta f\big)dV_g.
\]
Condition \eqref{eq:moment-cond} and Gronwall's inequality yield \eqref{eq:moment-contract}. For an individual trajectory, \eqref{eq:pointwise-distance} gives
$\delta_N(\Phi_t(x))\le e^{-2\eta t}\delta_N(x)$; taking square roots gives \eqref{eq:distance-exp}.
\end{proof}

\begin{remark}
The two conclusions of Theorem~\ref{th:distance} are logically distinct. Condition \eqref{eq:moment-cond} controls an integral quantity for transported sets. Pointwise exponential attraction to $N$ follows only from the additional differential inequality \eqref{eq:pointwise-distance} (or another genuine pointwise stability argument).
\end{remark}

\section{Conformal interpretation}\label{sec:geometry}

The density divergence has a simple differential-geometric interpretation. 
Define
$\widetilde g_\theta:=\rho_\theta^{2/n}g$.

\begin{proposition}
\label{prop:conformal}
The Riemannian volume form of $\widetilde g_\theta$ is
\begin{equation}\label{eq:vol-conformal}
dV_{\widetilde g_\theta}=\rho_\theta\,dV_g,
\end{equation}
and
\begin{equation}\label{eq:div-conformal}
\divg_{\widetilde g_\theta}f
=\divg_gf+\Lie_f\ln\rho_\theta
=D_\theta f.
\end{equation}
Thus weighted phase-volume contraction in $(M,g)$ is ordinary Riemannian volume contraction in the conformal metric $\widetilde g_\theta$.
\end{proposition}

\begin{proof}
If a metric is multiplied by the scalar factor $\rho_\theta^{2/n}$, its determinant is multiplied by $\rho_\theta^2$, hence its volume form by $\rho_\theta$, proving \eqref{eq:vol-conformal}. The divergence with respect to a volume form $\rho_\theta dV_g$ is $\rho_\theta^{-1}\divg_g(\rho_\theta f)$, which equals \eqref{eq:div-conformal} by Proposition~\ref{prop:transport}.
\end{proof}

\begin{remark}
For the scalar family $\rho_\alpha=w^\alpha$, the conformal metric is
$\widetilde g_\alpha=w^{2\alpha/n}g$. In particular, $\alpha=1$ gives $w^{2/n}g$, not $wg$ unless $n=2$.
\end{remark}

\section{Analytically solvable two-parameter example}\label{sec:example}

Let $M=\R^2\setminus\{0\}$ with the Euclidean metric. In polar coordinates on $M$, consider the annulus
\begin{equation*}\label{eq:K-example}
K:=\{(r,\vartheta):0.65\le r\le1.35,\ \vartheta\in\mathbb S^1\}
\end{equation*}
and the system
\begin{equation*}\label{eq:cycle}
\dot r=r(1-r),\qquad
\dot\vartheta=4+3\sin\vartheta.
\end{equation*}
The radial vector field points into $K$ at both radial boundaries, so $K$ is positively invariant. The circle $N=\{r=1\}$ is an attracting periodic orbit because $4+3\sin\vartheta\ge1$ and $r(1-r)$ changes sign across $r=1$.

For a vector field written in the coordinate basis $(\partial_r,\partial_\vartheta)$, the Euclidean divergence is
\begin{equation*}\label{eq:div-example}
\divg f
=\frac1r\frac{\partial}{\partial r}(r\dot r)
+\frac{\partial\dot\vartheta}{\partial\vartheta}
=2-3r+3\cos\vartheta.
\end{equation*}
Its maximum on $K$ equals $3.05>0$, so ordinary phase volume does not contract uniformly.

Choose two fixed basis functions
\begin{equation*}\label{eq:basis-example}
\psi_1(r,\vartheta)=-\ln(4+3\sin\vartheta),
\qquad
\psi_2(r,\vartheta)=\ln r,
\end{equation*}
and the compact parameter box
$\ThetaSet=[0,2]\times[-4,0]$.
The density is
\begin{equation*}\label{eq:rho-example}
\rho_\theta(r,\vartheta)
=(4+3\sin\vartheta)^{-\theta_1}r^{\theta_2}.
\end{equation*}
The Lie derivatives of the basis functions are
$\Lie_f\psi_1=-3\cos\vartheta$,
$\Lie_f\psi_2=1-r$.
Hence
\begin{equation}\label{eq:D-example}
D_\theta f
=2-3r+3(1-\theta_1)\cos\vartheta
+\theta_2(1-r).
\end{equation}
Maximizing \eqref{eq:D-example} over $K$ gives
\begin{align}
q_K(\theta)
&=-\max_{(r,\vartheta)\in K}D_\theta f\notag\\
&=1-3|\theta_1-1|-0.35|\theta_2+3|.\label{eq:q-example}
\end{align}
Indeed, the angular maximum contributes $3|1-\theta_1|$, whereas the radial affine term attains the maximum
\[
\max\{0.05+0.35\theta_2,\,-2.05-0.35\theta_2\}
=-1+0.35|\theta_2+3|.
\]
Therefore the constrained optimizer is unique:
\begin{equation*}\label{eq:theta-star}
\theta^*=(1,-3)^\top,\qquad q_K^*=1.
\end{equation*}
At this point
\begin{equation}\label{eq:Dstar-example}
D_{\theta^*}f\equiv-1~~\text{on }K,
\end{equation}
so every measurable $D\subset K$ satisfies
\begin{equation*}\label{eq:vol-example}
\mu_{\theta^*}(\Phi_t(D))\le e^{-t}\mu_{\theta^*}(D).
\end{equation*}
The optimal density
\[
\rho_{\theta^*}(r,\vartheta)
=\frac{1}{(4+3\sin\vartheta)r^3}
\]
is smooth and positive on $K$. Its conditioning on $K$ is
\begin{equation*}\label{eq:cond-example}
\frac{\max_K\rho_{\theta^*}}{\min_K\rho_{\theta^*}}
=7\left(\frac{1.35}{0.65}\right)^3\approx62.7.
\end{equation*}
Because $r>0$ on $M$ and $4+3\sin\vartheta\ge1$, the basis functions and every density in the prescribed family are smooth on $M$. Thus the improved rate is not produced by a singular density on the certified region.

The second basis function is essential for this regional certificate. If only the angular basis is allowed, i.e. $\theta_2=0$ and $\theta_1\in[0,2]$, then \eqref{eq:q-example} reduces to
\begin{equation*}\label{eq:single-basis-fail}
q_K(\theta_1,0)=-0.05-3|\theta_1-1|<0.
\end{equation*}
Hence no member of this natural one-basis family certifies uniform weighted-volume contraction on $K$, whereas independent radial shaping yields the unit rate in \eqref{eq:Dstar-example}. This illustrates the practical difference between a prescribed scalar family and a genuinely multi-parameter basis.

Once $\theta^*$ has been found, one can of course rename the single optimized function $\rho_{\theta^*}$ as a base weight and write it as $w^1$. That post hoc reparameterization is tautological. It does not solve the design problem of discovering the combination from the prescribed basis functions. The role of the multi-parameter family is precisely to make that combination a finite-dimensional optimization variable.

The active-point condition of Proposition~\ref{prop:active} is visible explicitly. The optimizer lies in the interior of $\ThetaSet$, and all points of $K$ are active because $D_{\theta^*}f\equiv-1$. For
\[
x_1=(0.65,0),\qquad x_2=(1.35,\pi),
\]
we have
\[
b(x_1)=(-3,0.35)^\top,\qquad
b(x_2)=(3,-0.35)^\top,
\]
so $\tfrac12b(x_1)+\tfrac12b(x_2)=0$. Thus only two active points are needed, consistent with the bound $m\le p+1=3$.

Finally, the same example illustrates Section~\ref{sec:submanifold}. Let
\[
U:=\operatorname{int}K=\{(r,\vartheta):0.65<r<1.35\}.
\]
This is a positively invariant tubular neighborhood of the invariant circle $N=\{r=1\}$. On $U$,
\[
\delta_N=(r-1)^2,
\qquad
\Lie_f\delta_N=-2r\delta_N.
\]
At $\theta^*$, equation \eqref{eq:Dstar-example} gives
\[
\Lie_f\delta_N+\delta_ND_{\theta^*}f
=-(2r+1)\delta_N\le-2.3\delta_N,
\qquad x\in U.
\]
Therefore Theorem~\ref{th:distance} yields, for every measurable $D\Subset U$,
\begin{equation*}\label{eq:moment-example}
\mathcal M_{N,\theta^*}(\Phi_t(D))
\le e^{-2.3t}\mathcal M_{N,\theta^*}(D).
\end{equation*}
Independently, $\Lie_f\delta_N=-2r\delta_N\le-1.3\delta_N$ on $U$, so the pointwise estimate \eqref{eq:distance-exp} holds for every $x\in U$ with $\eta=0.65$. The two rates refer to different objects and should not be conflated.

\section{Conclusion}\label{sec:conclusion}

The proposed framework turns regional weighted phase-volume contraction into a finite-dimensional and computationally tractable density-design problem. The log-affine parameterization leads to a convex max--min formulation, while active-point, minimax, and sampling certificates provide rigorous tools for both optimization and verification. The two-parameter example demonstrates that an appropriately designed density can certify uniform contraction in regions where ordinary divergence and a single-basis density fail.

The same structural principle applies to discrete-time maps and admits a natural conformal-geometric interpretation. Distance-weighted moments further extend the approach to invariant submanifolds, while preserving the distinction between set-wise contraction and pointwise stability. These results provide a systematic way to construct and certify weighted-volume contraction properties for nonlinear systems and create a basis for robust, adaptive, and controller-dependent density design.

\bibliographystyle{unsrtnat}
\bibliography{cas-refs-SCL-ready-v2}

\end{document}